\documentclass[a4paper,reqno,oneside]{amsart}

\usepackage{amsmath,amsthm, amssymb,eufrak,yfonts}
\usepackage{geometry} 
\usepackage[utf8]{inputenc} 
\usepackage[T1]{fontenc} 
\usepackage[english]{babel} 
\usepackage{fix-cm}
\usepackage{nicefrac} 
\usepackage{hyperref}
\usepackage{graphicx}
\usepackage{braket}
\usepackage{color}
\usepackage[usenames,dvipsnames]{xcolor}
\usepackage{csquotes}
\usepackage{dsfont}
\usepackage[sc]{mathpazo}
\usepackage{subcaption}
\usepackage{comment}
\newtheorem{proposition}{Proposition}

\theoremstyle{definition}

\newcommand{\beq}{\begin{equation}}
\newcommand{\eeq}{\end{equation}}
\newcommand{\ben}{\begin{enumerate}}
\newcommand{\een}{\end{enumerate}}
\newcommand{\bit}{\begin{itemize}}
\newcommand{\eit}{\end{itemize}}
\newcommand{\dys}{\displaystyle}

\newcommand{\vfi}{\varphi} 

\providecommand{\abs}[1]{\left|#1\right|}

\renewcommand{\k}{\mathtt k}

\title[Corrigendum]{Corrigendum to "Spectral optimization for weighted anisotropic problems with Robin conditions" [J. Differ. Equ. 378, 303--338, 2024]}
\date{}
\author{Benedetta Pellacci}
\address[B. Pellacci]{Dipartimento di Matematica e Fisica,
Universit\`a della Campania  ``Luigi Vanvitelli'',  via A. Lincoln 5, 81100
Caserta, Italy.}
\email[B. Pellacci]{benedetta.pellacci@unicampania.it}
\author{Giovanni Pisante}
\address[G. Pisante]{Dipartimento di Matematica e Fisica,
Universit\`a della Campania  ``Luigi Vanvitelli'',  via A. Lincoln 5, 81100
Caserta, Italy.}
\email[G. Pisante]{giovanni.pisante@unicampania.it}
\author{Delia Schiera}
\address[D. Schiera]{Departamento de Matemática do Instituto Superior Técnico, 
Universidade de Lisboa, 
Av. Rovisco Pais, 
1049-001 Lisboa, Portugal.}
\email[D. Schiera]{delia.schiera@tecnico.ulisboa.pt}

\begin{document}
\maketitle
The goal of this note is to fill a gap in the proof of the first two items of \cite[Theorem 5.1]{PPS}, which relies on Polya type inequalities and the characterization of the equality cases for monotone rearrangements (\cite[Propositions 4.1 and 4.2]{PPS}) whose statements and proofs require some adjustments.
More precisely, at the beginning of the proof of \cite[Proposition 4.1]{PPS} we stated that we can assume w.l.o.g. that the maximum point of $u$ is at the origin, which is actually not true in the anisotropic setting. An inspection of the proof shows that this assumption is only used to infer that $\rho'_{1}(\lambda)<0$ for any $i\in \{1,k-1\}$ in order to prove the validity of inequality (4.11) by estimating the left hand side  from below with the addendum depending on $\rho_{1}'(\lambda)$. Actually this information is needed only for the indices $i$ such that $N(i)=1$. Indeed, as soon as there are at least two elements in the sum, we can choose another $j$ to estimate from below and we can obtain the needed estimate. We start by proving the following more refined estimates. 
\begin{proposition}\label{polya monotona augmented}
Let  $u \in W^{1,p}(0, 1)$. If $u(0) \geq u(1)$ then 
\begin{equation}\label{Polya-down-aug} 
 \int_0^1 H^p(u') \geqslant b^p \int_0^1 |(u_*)'|^p + a^p \int_{O_{*}} |(u_*)'|^p, 
 \end{equation}
where we have defined
\[
O_{*}:= \big\{ t\in (0,1) \;:\; u_{*}(t) > u(0) \big\} \cup \big \{ t\in (0,1) \;:\;u_{*}(t) < u(1)\big\}. 
\]
Otherwise, if $u(0) \leq u(1)$ then 
\begin{equation}\label{Polya-up-aug}  
 \int_0^1 H^p(u') \geqslant a^p \int_0^1 |(u^*)'|^p + b^{p} \int_{O^{*}} |(u^*)'|^p,
\end{equation} 
where we have defined 
\[
O^{*}:= \big\{ t\in (0,1) \;:\; u^*(t) > u(1) \big\} \cup \big\{ t\in (0,1) \;:\; u^*(t) <u(0)\big\}. 
\] 

 \end{proposition}
\begin{proof} The proof follows the line  of \cite[Proposition 4.1]{PPS}, therefore we use the same notation and highlight here only the differences. As already pointed out, to prove \eqref{Polya-down} or \eqref{Polya-up} we can assume w.l.o.g. that  $u > 0$ and piecewise affine but we cannot assume that $u$ has a maximum at the origin. As a consequence we don't know a priori that the sign of $\rho_{1}'$ is always negative and we can only infer that
 \[ 
\text{sign}\left(\rho_j'(\lambda) \right)= \text{sign}\left(u'(\rho_j(\lambda)) \right) = (-1)^{j+1}s_{i}, \;\;\; \text{ with }\;\; s_{i}:= \text{sign}\left(u'(\rho_1(\lambda)) \right) =  \text{sign}(\rho'_1(\lambda))\]
 and therefore \cite[equation (4.9) and (4.8)]{PPS} have to be replaced by 

\begin{equation*}\label{eq:rhostar}  \rho_* (\lambda)= 
\begin{cases}
\displaystyle  \frac{1+s_{i}}{2}+s_{i}\sum_{j=1}^{N(i)} (-1)^{j} \rho_j(\lambda)  &  \text{ if $N(i)$ is odd,}  \\
\displaystyle \frac{1-s_{i}}{2}+s_{i}\sum_{j=1}^{N(i)} (-1)^{j} \rho_j(\lambda)  & \text{ if $N(i)$ is even,} 
\end{cases} \end{equation*}
and
\begin{equation}\label{rho*'} 
\rho_*'(\lambda)=s_{i}\sum_{j=1}^{N(i)} (-1)^j  \rho_j'(\lambda)=-\sum_{j=1}^{N(i)}|\rho'_j(\lambda)|. 
\end{equation}
As a consequence, performing the change of variable $x=\rho_{*}(\lambda)$, we have
\[
\int_{E_i} |u'_*(x)|^{p}dx  =
\int_{a_i}^{a_{i+1}} \left(\sum_{j=1}^{N(i)} \abs{ \rho'_j (\lambda)} \right)^{1-p}  d\lambda.
\]
Moreover, we have
\[
\int_{D_i} H^p(u'(x))dx = \sum_{j=1}^{N(i)} \int_{a_i}^{a_{i+1}}  \gamma_{j}^{i} \abs{ \rho'_j (\lambda)} ^{1-p} d\lambda
\]
with 
\[
\gamma_{j}^{i}  := \frac{1-s_{i}(-1)^{j}}{2} a^{p}+\frac{1-s_{i}(-1)^{j+1}}{2} b^{p}.
\]
Therefore, recalling that for any $j=1,\dots, N(i) $ and $i=1,\dots,k$ it holds
\begin{equation*}\label{ineq} 
\abs{  \rho'_j (\lambda)}^{1-p} \geqslant \left(  \sum_{j=1}^{N(i)} \abs{\rho'_j(\lambda)} \right)^{1-p}.
\end{equation*}
We infer 
\begin{equation}
\int_{D_i} H^p(u'(x))dx \geq \left (\sum_{j=1}^{N(i)}  \gamma_{j}^{i} \right) \int_{a_i}^{a_{i+1}} {\left( \sum_{j=1}^{N(i)} \abs{\rho'_j (\lambda)} \right)^{1-p}} d\lambda
\end{equation} 
Now we observe that if $ \min \{u(0),u(1)\} \leq a_{i} < \max \{u(0),u(1)\}$, then $s_{i}=\text{sign}(u(1)-u(0))$ and $\gamma_{1}^{i} = b^{p}$ if $u(0)\geq u(1)$ and $\gamma_{1}^{i} = a^{p}$ if  $u(0)\leq u(1)$. On the other hand,  if $a_{i}\geq \max \{u(0),u(1)\}$ or $a_{i} < \min \{u(0),u(1)\}$, then $N(i)$ is even, and so 
\[
\sum_{j=1}^{N(i)}  \gamma_{j}^{i} \geq a^{p}+b^{p}.
\]
In the case $u(0)\geq u(1)$, the proof follows since, up to a set of measure zero, 
\[
O_{*} \cap \{u'\not= 0\} = \bigcup_{i \in I} (a_{i },a_{i +1}) \;,\;\; I=\big\{ i \;: \; a_{i}\geq \max \{u(0),u(1)\} \text{ or } a_{i} < \min \{u(0),u(1)\}\big \}.
\]
The other case is analogous up to defining the set $D_{i }$ using the monotone increasing rearrangement.
\end{proof}

As a consequence of the previous Proposition, we get the following Polya type inequalities whith a characterizaton of the equality cases that replaces \cite[Proposition 4.1]{PPS}. 

\begin{proposition}\label{polya monotona}
Let  $u \in W^{1,p}(0, 1)$. If $u(0) \geq u(1)$ then 
\begin{equation}\label{Polya-down} 
 \int_0^1 H^p(u') \geqslant  \int_0^1 H^p((u_{*})'); 
 \end{equation}
otherwise, if $u(0) \leq u(1)$ then 
\begin{equation}\label{Polya-up}  
 \int_0^1 H^p(u') \geqslant  \int_0^1 H^p((u^*)'). 
 \end{equation} 
 Moreover if equality is attained in \eqref{Polya-down} or \eqref{Polya-up}, then $u$ is monotone.
 \end{proposition}

\begin{proof}
We only need to justify the case of equality, since inequalities \eqref{Polya-down} and  \eqref{Polya-up} immediately follow from Proposition \ref{polya monotona augmented}. 

For the equality case, we first observe that if equality holds in \eqref{Polya-down}, by \eqref{Polya-down-aug} we infer that the set $O_{*}$ has measure zero. This ensure that, if $u$ is not constant, then 
\[
u(0) =   \max_{t\in [0,1]} u(t) > \min_{t\in [0,1]} u(t)   = u(1).
\] 
The proof follows exactly as in \cite[Proposition 4.2]{PPS} once we note that, given $\kappa$ as in \cite[(4.13)]{PPS},  the functions $u_{1}=u(t)$ for $t\in (0,\kappa)$, $u_{2}=u(t+\kappa)$ for $t\in (0, 1-\kappa)$ and 
\[ v(x):=\begin{cases}
(u_1)_{*}(x) &\text{ in } (0, \kappa)\\
(u_2)_{*}(x-\kappa) & \text{ in } (\kappa, 1).
\end{cases}\]
satisfy $u_{1}(0) \geq u_{1}(\kappa)$, $u_{2}(0)\geq u_{2}(1-\kappa)$ and $v(0)\geq v(1)$. 
Moreover, we observe that by a density argument Proposition \ref{polya monotona} can be applied to $v$  (see \cite[Remark 2.23a]{Kawohl}). The corresponding result, when equality holds in \eqref{Polya-up} and $u(0) \leq u(1)$, can be proved analogously by considering the auxiliary function 
\[ w(x):=\begin{cases}
u_1^*(x) &\text{ in } (0, \kappa)\\
u_2^*(x-\kappa) & \text{ in } (\kappa, 1).
\end{cases}\]
Instead of $v$ and applying \eqref{Polya-up}. 
\end{proof}

Another consequence that replace \cite[Proposition 4.2]{PPS} is the following.
\begin{proposition}\label{equality polya} 
Let  $u \in W^{1,p}(0, 1)$, then
\begin{equation} \label{polya:ineq}
\int_0^1 H^p(u') \geqslant \min\left\{ \int_0^1 H^p((u_*)'), \int_0^1 H^p((u^*)') 
\right\}
\end{equation}
moreover if equality holds in \eqref{polya:ineq} then $u$ is monotone.
 \end{proposition}

\begin{proof}
 As $u_*(x)=u^*(1-x)$ (see for example \cite[Section 1.4]{Kesavan}), denoting by
 \[
 K:= \int_0^1 |(u_*)'|^p = \int_0^1 |(u^*)'|^p,
 \] 
from Proposition \ref{polya monotona} we readily infer that 
\begin{equation*} 
\int_0^1 H^p(u') \geqslant \min\left\{ \int_0^1 H^p((u_*)'), \int_0^1 H^p((u^*)') 
\right\} = \min \{a^{p},b^{p}\} K. 
\end{equation*}
For the equality case, if $a>b$, the proof goes as in \cite[Proposition 4.2]{PPS}, up to using inequality \eqref{polya:ineq} instead of applying  \cite[Proposition 4.1]{PPS}. In the case $a < b$ we need to substitute decreasing rearrangements with increasing ones. 
\end{proof}

We are now ready to give a correct proof of conclusons $1.$ and $2.$ of \cite[Theorem 5.1]{PPS}.

\begin{proof}[Proof of conclusion $1$]
Let $a > b$, the other case follows analogously. Notice that, by elliptic regularity, $\vfi'\in C([0,1])$, so that,
if $\k>0$ the boundary conditions immediately imply that $\vfi$ achieves its maximum in $(0,1)$.
Let $\alpha>0$ be the first global maximum of $\vfi$. 

Let us define
\[ \vfi^R(x):=\begin{cases}
\vfi^*(x) & \text{ in } (0, \alpha)\\
\vfi_*(x) &\text{ in } (\alpha, 1), 
\end{cases} 
\; \text{ and }\;
 m^R(x):=\begin{cases}
m^*(x) & \text{ in } (0, \alpha)\\
m_*(x) &\text{ in } (\alpha, 1). 
\end{cases}\]
On $(0, \alpha)$, since $\alpha$ is a global maximum, we can use the second conclusion of Proposition \ref{polya monotona} to get
\[ \int_0^\alpha H^p(\vfi') \geqslant \int_0^\alpha H^p((\vfi^*)') = \int_0^\alpha H^p((\vfi^R)'). \]
On the other hand, the first conclusion of  Proposition \ref{polya monotona} yields
\[ \int_\alpha^1 H^p(\vfi') \geqslant \int_\alpha^1 H^p((\vfi_*)')= \int_\alpha^1 H^p((\vfi^R)').\]
Then, arguing as in \cite[Theorem 5.1]{PPS}, we infer
\begin{equation*}\label{ug crescente} \int_0^\alpha H^p(\vfi') = \int_0^\alpha H^p((\vfi^*)'), \quad \int_\alpha^1 H^p(\vfi') = \int_\alpha^1 H^p((\vfi_*)'). \end{equation*}
Therefore, we can apply Proposition \ref{polya monotona} to get that $\vfi$ is increasing in $(0,\alpha)$ and decreasing in $(\alpha, 1)$.
\end{proof}

\begin{proof}[Proof of conclusion $2$]
Let $a > b$ and consider the decreasing rearrangements $\varphi_*$ and $m_*$. 
Then, by Proposition \ref{equality polya}
\[ 
\Lambda^+= \frac{\int_0^1 H(\varphi')^p }{ \int_0^1 m(x) \varphi^p} \geqslant\frac{\int_0^1 H((\varphi_*)')^p }{ \int_0^1 m_*(x) \varphi_*^p}
 \geqslant \Lambda^+. 
\]
Therefore, equality holds, and by the equality case in Proposition \ref{equality polya} one obtains that $\varphi$ is monotone (decreasing). 
The case $a < b$ follows similarly, up to considering increasing rearrangements. 
\end{proof}

\end{document}